\documentclass[12pt]{article}
\usepackage[utf8]{inputenc}
\usepackage{amsmath,amssymb,amsthm}
\renewcommand\le{\leqslant}
\renewcommand\ge{\geqslant}

\newtheorem{theorem}{Theorem}
\newtheorem{lemma}{Lemma}
\newtheorem*{theoremA}{Theorem A}
\theoremstyle{remark}
\newtheorem*{remark}{Remark}

\newcommand\R{\mathbb R}
\newcommand\E{\mathsf E}

\renewcommand\P{\mathsf P}

\newcommand\Cc{\mathcal C}
\newcommand\pOmega{\mathsf{\Omega}}
\newcommand{\lqw}{\ell_{q,w}^N}
\newcommand{\lw}[1]{\ell_{{#1},w}^N}
\DeclareMathOperator{\extr}{extr}
\DeclareMathOperator{\supp}{supp}

\title{Kolmogorov widths of Besov classes $B^1_{1,\theta}$ and products of octahedra}
\author{Yuri Malykhin}

\begin{document}
\maketitle
\abstract{In this paper we find the orders of decay for Kolmogorov widths of
some Besov classes related to $W^1_1$ (the behaviour of the widths for $W^1_1$
remains unknown):
$$
d_n(B^1_{1,\theta}[0,1],L_q[0,1])\asymp n^{-1/2}\log^{\max(\frac12,1-\frac{1}{\theta})}n,\quad
2<q<\infty.
$$
The proof relies on the lower bound for widths of product of octahedra in a
special norm (maximum of two weighted $\ell_q$ norms).
This bound generalizes the theorem of B.S.~Kashin on widths of octahedra
in $\ell_q^N$.
}

\section{Introduction}

Let us recall the definition of widths.
Let $X$ be a normed space and $K\subset X$. The Kolmogorov $n$-width of $K$ in
$X$ is the following quantity
$$
d_n(K,X) := \inf_{\substack{Q_n\subset X\\\dim Q_n\le n}} \sup_{x\in
K}\inf_{y\in Q} \|x-y\|_X,
$$
where $\inf_{Q_n}$ is taken over all linear subspaces in $X$ having dimension at
most $n$.

The main problem of the widths theory is to find the order of decay of
$d_n(W,X)$ at $n\to\infty$ for functional classes and finite-dimensional sets. 
One of interesting cases is the family of Sobolev spaces $\mathcal W^r_p$ consisting of
functions $f$ with a finite norm $\|f\|_p + \|f^{(r)}\|_p$ and Sobolev classes
$W^r_p$ (unit balls in corresponding spaces).
The order of $d_n(W^r_p[0,1],L_q[0,1])$ is known for all
$r\in\mathbb N$ and $1\le p,q\le\infty$ (V.M.~Tikhomirov, B.S.~Kashin,
E.D.~Gluskin, R.S.~Ismagilov~--- see~\cite{LGM}, Chapters 13, 14, and
the survey~\cite{TDU}, \S4.3), except the case $r=1$, $p=1$,
$q\in(2,\infty)$. There is a logarithmic gap~\cite{Kul,KMR}:
$$
    c(q)n^{-1/2}\log^{1/2}n \le d_n(W^1_1[0,1],L_q[0,1]) \le C(q)n^{-1/2}\log n.
$$
The order of approximation is known only for special approximation subspaces
$Q_n=\mathrm{span}\{\exp(im_jx)\}_{j=1}^n$, see.~\cite{Belinsky1987}:
$d_n^T(W^1_1[0,1],L_q[0,1]) \asymp n^{-1/2}\log n$.

Sobolev spaces are related to Besov spaces. Let us give the definition,
following~\cite{DL}, Chapter 2, \S10. 
Given parameters $r>0$, $p,\theta\in[1,\infty]$ and a (finite or infinite)
segment $I\subset\mathbb R$, we define the Besov space
$\mathcal B^r_{p,\theta}(I)$ as the space of functions
 $f\in L_p(I)$ having a finite semi-norm (with $s=1+\lfloor
r\rfloor$)
$$
    |f|_{B^r_{p,\theta}(I)} :=
    \begin{cases}
        \int_0^\infty (t^{-r}\omega_s(f,t)_p^\theta)^{1/\theta} \frac{dt}{t},&\quad 1\le \theta < \infty,\\
        \sup_{t>0}t^{-r}\omega_s(f,t)_p,&\quad \theta=\infty.
    \end{cases}
$$
Here $\omega_s(f,t)_p := \sup_{0<h\le
t}\|\Delta_h^s(f,\cdot)\|_{L_p(I\cap(I-sh))}$ is the usual modulus of
smoothness. The norm in the Besov space is defined as
$\|f\|_{L_p(I)}+|f|_{B^r_{p,\theta}(I)}$. 
The Besov class $B^r_{p,\theta}(I)$ is the unit ball in the corresponding space.

\begin{remark}
    One can define Besov space on domains $I\ne\R$ as the restriction of
    functions $f\in B^r_{p,\theta}(\mathbb R)$ to $I$. Then the norm of $g$
    is defined as
$\inf\{\|f\|\colon \left.f\right|_I=g\}$. For segments $I\subset\R$ these
    definitions (and corresponding norms) are equivalent, see, e.g.,~\cite{Triebel3}, \S1.11, Theorem 1.118.
\end{remark}


There are embeddings $\mathcal B^r_{p,\min(p,2)}
\hookrightarrow \mathcal W^r_p \hookrightarrow \mathcal B^r_{p,\max(p,2)}$ for
$1<p<\infty$ and $\mathcal B^r_{p,1} \hookrightarrow \mathcal W^r_p
\hookrightarrow \mathcal B^r_{p,\infty}$ for $p=1$, $p=\infty$. As a corollary,
we have $c_1 B^1_{1,1}[0,1] \subset W^1_1[0,1] \subset
c_2 B^1_{1,\infty}[0,1]$. This gives us motivation to study widths of
$B^1_{1,\theta}$.

Of course, widths of Besov classes are interesting by themselves.
Let us mention the papers of Ding Dung~\cite{Dung}, A.S.~Romanyuk~\cite{Romanyuk},
E.M.~Galeev~\cite{Galeev2001}; see also the survey~\cite{TDU}, \S4.3.
The case $B^1_{1,\theta}$ in $L_q$, $q>2$, as far as author knows, was not
studied before.

\begin{theorem}\label{th_besov}
    For any $q\in(2,\infty)$ there exist $c_{q}>0$ and $C_{q}>0$, such that
    for all $1\le\theta\le\infty$ and $n=2,3,\ldots$ the inequalities hold:
    $$
    c_qn^{-1/2}\log^{1/2}n \le d_n(B^1_{1,\theta}[0,1],L_q[0,1])\le
    C_qn^{-1/2}\log^{1/2}n,\quad 1\le\theta\le 2,
    $$
    $$
    c_qn^{-1/2}\log^{1-1/\theta}n \le d_n(B^1_{1,\theta}[0,1],L_q[0,1])\le
    C_qn^{-1/2}\log^{1-1/\theta} n,\quad 2\le\theta\le\infty.
    $$
\end{theorem}

Unfortunately, this gives us no new information about $W^1_1$.

The proof is based on the following finite-dimensional result.
Consider a weighted norm in $\R^N$:
$$
\|x\|_{\lqw} := \left\{\sum_{i=1}^N w_i|x_i|^q\right\}^{1/q}.
$$
Suppose that the set of coordinates is divided into $m$ blocks:
$\{1,\ldots,N\}=\sqcup_{s=1}^m \Delta_s$, $|\Delta_s|=N_s$.
Let us take the octahedron $B_1^{N_s}$ in each block, and consider their
Cartesian product
$$
\prod_{s=1}^m B_1^{N_s} = \{x\in \R^N\colon \sum_{i\in\Delta_s} |x_i| \le
1,\;s=1,\ldots,m\}.
$$

\begin{theorem}\label{th_w}
    Let there are weights $w_1,\ldots,w_N\ge 0$, division into $m$ blocks
    $\{1,\ldots,N\}=\sqcup_{s=1}^m \Delta_s$, $N_s=|\Delta_s|$, and a number
    $n\in\mathbb N$, such that the following conditions hold:
    \begin{itemize}
        \item[1)] the weights are sufficiently small: $\max\limits_{1\le i\le N} w_i \le
            (4n)^{-1}\sum_{i=1}^N w_i$;
        \item[2)] the weighted ratio of each block is not small:
            $$
            \nu_s := \frac{\sum_{i\in\Delta_s}w_i}{\sum_{i=1}^N w_i} \ge
            \frac{C\log 2m}{n},\quad s=1,\ldots,m.
            $$
        where $C$ is enough large absolute constant.
    \end{itemize}
    Suppose that for $q\in[2,\infty)$ and $h>0$ the norm in the space
    $X=(\R^N,\|\cdot\|)$ satisfies
    $$
    \|x\|_X \ge \max(\|x\|_{\lw{q}},h\|x\|_\infty),\quad x\in\R^N.
    $$
    Then the Kolmogorov width of the product of octahedra in $X$ satisfies the
    inequality
$$
    d_n(\prod_{s=1}^m B_1^{N_s},X)\ge
    \min(c_q\,(n\max_{1\le s\le m}\nu_s)^{-1/2}(\sum_{i=1}^N w_i)^{1/q}, h/2).
$$
\end{theorem}

We note that the condition 2) implies that there are not so many blocks:
$Cm\log(2m)\le n$.

Let us mention some corollaries of the Theorem 2. Here and after by intersection
$X\cap Y$ of two normed spaces we mean the space with the norm
$\|x\|_{X\cap Y} := \max(\|x\|_X,\|x\|_Y)$.
If we take $q_2\ge 1$, then $\|x\|_{\lw{q_2}}\ge \min_i w_i^{1/q_2}\|x\|_\infty$.
Suppose that the sum of weights in each block is constant
(hence $\nu_s=m^{-1}$), then the bound in the theorem becomes (where we put
$q_1:=q$ for more symmetry):
\begin{equation}
    \label{th_w_form}
    d_n(\prod_{s=1}^m B_1^{N_s},\lw{q_1}\cap\lw{q_2})\ge
    \min(c_{q_1}\sqrt{\frac{m}n}(\sum_{i=1}^N w_i)^{1/q_1}, \frac12\min w_i^{1/q_2}).
\end{equation}
We will use the bound in the form~(\ref{th_w_form}) to prove Theorem 1; we shall
take $q_1=2$ and $q_2$ equal to $q$ from Theorem 1.

If $q_1=q_2=q\in[2,\infty)$ in~(\ref{th_w_form}) then we get lower bound
for the width of product of octahedra in $\lqw$. Moreover, if all blocks are equal
and $w_i=1$, then we get
$$
d_n(\prod_{s=1}^m B_1^{N/m},\ell_q^N) \ge \min(c_q\sqrt{\frac{m}n}N^{1/q},\frac12),
\quad\mbox{if $Cm\log(2m)\le n\le N/4$}.
$$
For $m=1$ this was proven in the paper of B.S.~Kashin~\cite{K80}, there was also
noted that the bound is sharp:
\begin{equation}
    \label{octahedra}
    \frac14\min(\frac{N^{1/q}}{\sqrt{n}},1) \le d_n(B_1^N,\ell_q^N) \le
    \min(c_q\frac{N^{1/q}}{\sqrt{n}},1),\quad n<\frac{N}{2},\;q\in(2,\infty).
\end{equation}
Our proof of~\ref{th_w} is the development of a method from~\cite{K80}.

The product of octahedra had already appeared in the context of widths of
functional classes, see, e.g.~\cite{Galeev90}, \cite{MR17}.

\section{Proofs}
We will prove Theorem~\ref{th_w} and then derive Theorem~\ref{th_besov}.
\begin{proof}
W.l.o.g. $\|x\|_X=\max(\|x\|_{\lw{q}},h\|x\|_\infty)$.
We denote $B:=\prod_{s=1}^m B_1^{N_s}$,
$d_n := d_n(B,X)$, and let $L_n$ be the extremal subspace for that width.
If $d_n\ge h/2$, the proof is done, so we will assume that $d_n<h/2$.

We are interested in the approximation of the extreme points of $B$; it is easy to see that
$\extr B = \{\sum_{s=1}^m \pm e_{i_s}\colon i_s\in\Delta_s\}$. Let us show that
for any $u\in\extr B$ there is $v=v(u)\in L_n$ such that
\begin{equation}\label{1etap_a}
    \|u-v\|_X\le 2d_n,
\end{equation}
\begin{equation}\label{1etap_b}
    u_i=v_i\quad\forall\,i\in\supp u.
\end{equation}
Fix $u\in\extr B$ and take $v_0\in L_n$, such that $\|u-v_0\|_X\le d_n$.
Write the error $u-v_0$ as $u_1+\beta_1$, where
$\supp u_1 \subset\supp u$, $\supp\beta_1\cap\supp u=\varnothing$.
    We have $\|\beta_1\|\le d_n$,
    $$
    \|u_1\|_\infty \le h^{-1}\|u_1\|_X \le d_n/h \le 1/2,
    $$
    and the condition on $\supp u_1$ implies that $u_1\in\frac12B$. Hence, the
    vector $u_1$ may be approximated by $v_1\in L_n$ with an error $\le\frac12 d_n$.
Further, we write $u_1-v_1=u_2+\beta_2$, $\supp u_2\subset\supp u$ and approximate
$u_2$. If we repeat this process, we will get the approximation of $u$ by the sum
$v_0+v_1+\ldots\in L_n$ with the error $\|\beta_1+\beta_2+\ldots\|\le 2d_n$.

    Let us fix a mapping $v(\cdot)\colon \extr B\to L_n$,
    which satisfies~(\ref{1etap_a}) and~(\ref{1etap_b}).
    
    We can normalize the weights by $\sum_{i=1}^N w_i=1$.
Let $p_i:=2nw_i$. Let $\pOmega\subset\{1,\ldots,N\}$ be a random subset, such
    that any point $i$ gets in $\pOmega$ independently of others with
    probability $p_i$. Condition 1) of the Theorem gives us
$p_i\in[0,1]$ (even $p_i\le 1/2$; we will use that later). The choice of
    the probabilites proportional to the weights allows us to write weighted the norm as the
    expectation:
    $$
    \E\sum_{i\in\pOmega}|x_i|^{q} = 2n\|x\|_{\lw{q}}^{q}.
    $$
    Also, $\E|\pOmega|=\sum p_i=2n$ and
    $\E|\pOmega\cap\Delta_s|=2n\nu_s$.

    Let $\mathsf{u}$ be a random point in $\extr B$ (equiprobable) independent
    of $\pOmega$. In other words, in each block $\Delta_s$ one
    random coordinate of $\mathsf{u}$ equals $\pm1$ and others are zero.
    Put $\mathsf{v}=v(\mathsf{u})$.
    Consider the random variable
$$
    \xi := \sum_{i\in\pOmega}
    |\mathsf{u}_i-\mathsf{v}_i|^{q}\cdot\mathbf{1}\{\supp \mathsf{u}\subset
    \pOmega\}
$$
and the conditional expectation
$$
    K := \E(\sum_{i\in\pOmega} |\mathsf{u}_i-\mathsf{v}_i|^{q}\,|\, \supp
    \mathsf{u}\subset \pOmega) =
    \frac{\E\xi}{\P(\supp \mathsf{u}\subset\pOmega)}.
$$
The meaning of the condition
$\{\supp\mathsf{u}\subset\pOmega\}$ is that it allows us to come to a
low-dimensional case: $|\pOmega|\approx\E|\pOmega| = 2n$.

We will bound $K$ in two ways.
    Fix $u^\circ$ and average over $\pOmega$:
$$
\E(\xi|\mathsf{u}=u^\circ) = \sum_{i=1}^N|u^\circ_i-v^\circ_i|^{q}\cdot
\P(\supp u^\circ\subset \pOmega,\;i\in \pOmega).
$$
If $i\in\supp u^\circ$, then by~(\ref{1etap_b}) the corresponding term is zero.
Otherwise, by the independence,
$$
    \frac{\P(\supp u^\circ\subset\pOmega,\;i\in\pOmega)}{\P(\supp
    u^\circ\subset\pOmega)}= P(i\in\pOmega)=2nw_i.
$$
We sum over $i$ and obtain
$$
\frac{\E(\xi|\mathsf{u}=u^\circ)}{\P(\supp u^\circ\subset\pOmega)} =
2n\sum_{i=1}^Nw_i|u^\circ_i-v^\circ_i|^{q} =
    2n\|u^\circ-v^\circ\|_{\lw{q}}^{q} \ll n d_n^{q}.
$$
Then we multiply by the probability, average over $u^\circ$ and obtain
$$
\E\xi \ll n d_n^{q}\cdot\P(\supp\mathsf{u}\subset\pOmega),\quad K \ll n d_n^{q}.
$$

Now we fix $\Omega^\circ$ and average over $u$: 
$$
    \E(\xi|\pOmega=\Omega^\circ) =
    \E\left(\left.\sum_{i\in\Omega^\circ}|\mathsf{u}_i-\mathsf{v}_i|^{q}\;\right|\supp
    \mathsf{u}\subset\Omega^\circ\right)\cdot\P(\supp \mathsf{u}\subset\Omega^\circ).
$$

Let us call a set $\Omega$ \textit{regular} if the cardinalities of the sets
$\Omega_s := \Omega\cap\Delta_s$ satisfy 
    $\frac32 n\nu_s\le |\Omega_s|\le An\nu_s$ for all $s$; the constant $A$
    is large enough (see later).
    
    Suppose that $\Omega^\circ$ is regular. Obviously, $\frac32n \le
    |\Omega^\circ|\le An$.
    By $\widetilde{\mathsf{u}}$ we denote the random vector with the conditional
    distribution
    $\mathsf{Law}(\widetilde{\mathsf{u}})=\mathsf{Law}(\mathsf{u}|\supp
    \mathsf{u}\subset\Omega^\circ)$. Let
    $\widetilde{\mathsf{v}} =
    v(\widetilde{\mathsf{u}})$. We have
    $$
\E\left(\left.\sum_{i\in\Omega^\circ}|\mathsf{u}_i-\mathsf{v}_i|^{q}\;\right|\supp
    \mathsf{u}\subset\Omega^\circ\right) =
    \E\|\widetilde{\mathsf{u}}-\widetilde{\mathsf{v}}\|_{\ell_{q}(\Omega^\circ)}^{q}.
    $$
    We can consider $\widetilde{\mathsf{u}}$ as a random vector in
    $\R^{\Omega^\circ}$.
    It is easy to see that $\widetilde{\mathsf{u}}$ takes values (equiprobable)
in the set $\extr\prod_{s=1}^m
B_1^{|\Omega_s^\circ|}$; in other words, in each set $\Omega^\circ_s$ one random
coordinate is $\pm1$ and others are zero.
We can pass from $q$-mean to $2$-mean and from
$\ell_{q}(\Omega^\circ)$ metric to $\ell_2(\Omega^\circ)$
(here we use that $q\ge 2$):
$$
    (\E \|\widetilde{\mathsf{u}}-\widetilde{\mathsf{v}}\|_{q}^{q})^{1/{q}} \ge
    (\E \|\widetilde{\mathsf{u}}-\widetilde{\mathsf{v}}\|_{q}^{2})^{1/{2}} \gg
    n^{1/{q}-1/2}(\E \|\widetilde{\mathsf{u}}-\widetilde{\mathsf{v}}\|_2^2
    )^{1/2}.
$$
The estimate for the approximation in Euclid metric is standard (a variation of
Eckart-Young theorem).

\begin{lemma}[See, e.g.,~\cite{Ismagilov}]
Let $\mathsf{z}$ be a random vector in Hilbert space $H$ with
$\E\|z\|_H^2<\infty$ and in some orthonormal basis
$\{\varphi_i\}$ we have $\mathsf{E}\langle \mathsf{z},\varphi_i\rangle\langle
\mathsf{z},\varphi_j\rangle = 0$ for $i\ne j$.
Let $\sigma_i := (\mathsf{E}\langle \mathsf{z},\varphi_i\rangle^2)^{1/2}$ and
    $\sigma_1^*\ge \sigma_2^*\ge\ldots$ be the non-increasing permutation of
    $(\sigma_i)$.
Then for any $n$-dimensional subspace $L_n\subset H$ the inequality holds
$$
    \E \rho(\mathsf{z},L_n)_H^2 \ge \sum_{k>n}(\sigma_k^*)^2,\quad
\mbox{where $\rho(x,L_n)_H := \inf_{y\in L}\|x-y\|_H$}.
$$
\end{lemma}

In our case one can take the standard basis $\{e_i\}$, because
$\E\widetilde{\mathsf{u}}_i\widetilde{\mathsf{u}}_j=0$;
$\sigma_i^2=\E\widetilde{\mathsf{u}}_i^2=|\Omega_s^\circ|^{-1} \asymp (n\nu_s)^{-1}$,
$i\in\Delta_s$. Denote $\nu := \max \nu_s$, then
$$
\E\|\widetilde{\mathsf{u}}-\widetilde{\mathsf{v}}\|_2^2 \ge
\E\rho(\widetilde{\mathsf{u}},L_n)_2^2 \ge \sum_{k=n+1}^{3n/2}(\sigma_k^*)^2 \gg
\nu^{-1}.
$$

Therefore, for regular $\Omega^\circ$ we have
$$
\E(\xi|\pOmega=\Omega^\circ) \gg \P(\supp \mathsf{u}\subset\Omega^\circ)\cdot
n^{1-q/2}\cdot \nu^{-q/2}.
$$
Averaging over $\Omega^\circ$ we obtain
$$
    \E\xi \gg n^{1-q/2}\nu^{-q/2}\cdot 
    \P(\supp \mathsf{u}\subset\pOmega,\;\pOmega\mbox{ is regular}).
$$
Later we will show that one can separate conditions in $\P(\cdot)$:
\begin{equation}\label{corr}
    \P(\supp \mathsf{u}\subset\pOmega,\;\pOmega\mbox{ is regular})\ge
    \P(\supp \mathsf{u}\subset\pOmega)\P(\pOmega\mbox{ is regular}).
\end{equation}
It is easy to check that the probability that $\pOmega$ is regular is
$\ge 1/2$. To do this we apply Bernstein's inequality:
$$
\P(|\sum_{i=1}^k X_i|>t)\le 2\exp(-\frac{t^2}{2(\sigma^2+Mt/3)}),
$$
where $X_1,\ldots,X_k$ are independent r.v. with zero average,
$M:=\max_i \|X_i\|_\infty$, $\sigma^2 := \sum_i \E X_i^2$. Fix $s$,
put
$X_i=\mathbf{1}\{i\in\pOmega_s\}-p_i$, $i\in\Delta_s$ (recall that
$\pOmega_s:=\pOmega\cap\Delta_s$), use that $M\le 1$, $\sigma^2\le
\sum_{i\in\Delta_s} p_i = 2n\nu_s$ to obtain
$$
\P(||\pOmega_s|-2n\nu_s|>n\nu_s/2)\le 2\exp(-\frac{(n\nu_s/2)^2}{2(2n\nu_s+n\nu_s/6)})\le
2\exp(-cn\nu_s).
$$
The condition 2) of the Theorem implies that this probability is less than $\frac{1}{2m}$.
Hence, the probability of irregularity of $\pOmega$ is less than $1/2$.
So,
$$
\E\xi \gg n^{1-q/2}\nu^{-q/2}\P(\supp \mathsf{u}\subset\pOmega),\quad K \gg
n^{1-q/2}\nu^{-q/2}.
$$

Finally, we compare upper and lower bounds on $K$ and obtain the required
bound for $d_n$.

It remains to prove~(\ref{corr}). We will use one technical lemma.
Here we write $\#\Omega$ instead of $|\Omega|$ for convenience.
\begin{lemma}
    Let $\delta\in(0,1)$, $\pOmega$ is a random subset of $\{1,\ldots,N\}$, and
    each point $i$ gets in $\pOmega$ independently of others with probability $p_i\le
    1/2$. If $\omega := \E\#\pOmega = \sum_{i=1}^N p_i$
    is large enough ($\omega \ge c(\delta)$), then
    $$
    \E(\#\pOmega|\mathcal R)\ge \E\#\pOmega,\quad\mbox{where $\mathcal
    R:=\{(1-\delta)\omega \le \#\pOmega \le 5\omega\}$}.
    $$
\end{lemma}

Let us derive~(\ref{corr}) from lemma. It is enough to prove the inequality for each
block (because of the independence):
$$
\P(\supp\mathsf{u}\cap\Delta_s \subset \pOmega\cap\Delta_s \big| 3n\nu_s/2 \le
|\pOmega_s|\le 5n\nu_s) \ge \P(\supp\mathsf{u}\cap\Delta_s \subset
\pOmega\cap\Delta_s).
$$
The right side of the inequality is $N_s^{-1}\E|\pOmega\cap\Delta_s|$, and left
side is equal to the corresponding conditional expectation. It remains to apply the Lemma for
$\pOmega\cap\Delta_s$.

Let us prove this lemma. The probability space is a union of non-intersecting
events $\mathcal R$, $\Cc_- :=\{\#\pOmega|<(1-\delta)\omega\}$ and
$\Cc_+ :=\{\#\pOmega>5\omega\}$. 
Denote $\Cc=\Cc_-\cup \Cc_+$. We have to prove that
$\E(\#\pOmega|\mathcal R)\ge \E(\#\pOmega|\Cc)$.
Using Bernstein's inequality we get
$$
\P(\#\pOmega \in [(1-\delta/3)\omega, (1+\delta/3)\omega])\ge
1-2\exp(-c_\delta\omega)>9/10 
$$
for large $\omega$. This is also true conditioning on $\mathcal R$, so
$$
\E(\#\pOmega|\mathcal R)\ge \frac{1}{10}(1-\delta)\omega +
\frac9{10}(1-\delta/3)\omega > (1-\delta/2)\omega.
$$

Let us bound $\E(\#\pOmega|\Cc) = \E(\#\pOmega|\Cc_-)\P(\Cc_-|\Cc) +
\E(\#\pOmega|\Cc_+)\P(\Cc_+|\Cc)$.

The first term is not more than $\E(\#\pOmega|\Cc_-)\le (1-\delta)\omega$.
The second term is equal to
$$
\E(\#\pOmega|\Cc_+)\P(\Cc_+|\Cc) = 
\frac{\E\#\pOmega\mathbf{1}_{\Cc_+}}{\P(\Cc)} = \frac{1}{\P(\Cc)}
\sum_{k>5\omega}k\P(\#\pOmega=k).
$$
Note that $\P(\Cc)\ge \P(\pOmega=\varnothing)=\prod(1-p_i)\ge \prod 4^{-p_i} =
4^{-\omega}$ (we use that $p_i\le 1/2$), hence the second term is less or equal than
$$
4^{\omega} \sum_{k>5\omega}k\P(\#\pOmega=k)\le
4^\omega\sum_{k>5\omega}k\exp(-c k)\le C\omega e^{-\omega} < \frac{\delta}2\omega
$$
for large $\omega$.
Summing up, we obtain $\E(\#\pOmega|\mathcal C)\le(1-\delta/2)\omega<\E(\#\pOmega|\mathcal R)$.
The lemma and the Theorem are proven.
\end{proof}

Before we proceed to the proof of Theorem 1, we will formulate some useful
statements.

It is well known that Besov spaces are isomorphic to some sequence spaces
(Lizorkin, Triebel, Meyer). Let us describe this isomorphism in the case of
spaces on $\R$. Let
$\psi$, $\varphi$ be a wavelet function and a scaling function, correspondingly.
We will make use of regular Daubechies wavelets with compact support
(see.~\cite{Triebel}, \S1.2): $\psi,\varphi\in
C^{r_0}(\mathbb R)$ and $\int \psi(x)x^j\,dx=0$, $0\le j<r_0$; with some fixed
$r_0\in\mathbb N$.
        
Denote $\psi_{k,j}(x):=2^{k/2}\psi(2^kx-j)$,
$\varphi_{0,j}(x):=\varphi(x-j)$, $k,j\in\mathbb Z$.

    We consider the space $\ell^\sigma_{p,\theta}$ of sequences 
    $(\lambda_{k,j})_{k\ge 0,j\in\mathbb Z}$ with the norm
$$
    \|(\lambda_{k,j})_{k\ge0,j\in\mathbb Z}\|_{\ell^\sigma_{p,\theta}} =
    \begin{cases}
        \left(\sum_{k\ge 0} 2^{k\sigma\theta}\|(\lambda_{k,j})_{j\in\mathbb
        Z}\|_{\ell_p}^\theta\right)^{1/\theta},&\quad \theta<\infty,\\
        \sup_{k\ge 0} 2^{k\sigma}\|(\lambda_{k,j})_{j\in\mathbb Z}\|_{\ell_p},&\quad
        \theta=\infty.
    \end{cases}
$$

    \begin{theoremA}[See.~\cite{Triebel}, \S1.2, Theorem 1.20]
    Suppose $0<r<r_0$; $p,\theta\in[1,\infty]$. Then $f\in L_p(\R)$ lies in
    $\mathcal{B}^r_{p,\theta}(\R)$ if and only if the following norm is finite:
    \begin{equation}\label{besov_discr}
    \|f\| := \|(\langle f,\varphi_{0,j}\rangle)_{j\in\mathbb Z}\|_{\ell_p} +
    \|(\langle f,\psi_{k,j}\rangle)_{k\ge 0,j\in\mathbb
        Z}\|_{\ell^\sigma_{p,\theta}},\;\sigma:=r+\frac12-\frac{1}{p}.
    \end{equation}
    This norm is equivalent to the norm of $\mathcal{B}^r_{p,\theta}(\R)$.
\end{theoremA}

    See also~\cite{HS}, Th.1.13.

    We will use later that Besov classes grow as $\theta$ increases
    ($B^1_{1,\theta_1}$ is embedded in $B^1_{1,\theta_2}$ as $\theta_1<\theta_2$), it is easily
    seen from this theorem.

    We should discretize the norm, too.
    Consider the subspace in $L_p$:
    $$
    L_p\langle\Psi_+\rangle := \overline{\mathrm{span}}\,\Psi_+,\quad
    \Psi_+:=\{\psi_{k,j}\colon k\ge 0,\,j\in\mathbb Z\}.
    $$
    (the closure is in $L_p(\R)$). 
    \begin{lemma}
        Let $f\in L_p\langle\Psi_+\rangle$. Then for $2\le p<\infty$ we have
        \begin{equation}
            \label{lq_discr}
        \|(\langle f,\psi_{k,j}\rangle)_{k\ge0,j\in\mathbb Z}\|_{\ell^{1/2-1/p}_{p,p}} \ll
        \|f\|_p \ll \|(\langle f,\psi_{k,j}\rangle)_{k\ge0,j\in\mathbb
        Z}\|_{\ell^{1/2-1/p}_{p,2}}.
        \end{equation}
    \end{lemma}

    We provide the proof for completeness; it is possible that such bounds are
    known. Similar inequalities had appeared in the context of Kolmogorov widths
    in, e.g.,~\cite{Kul} for Haar wavelets.

    \begin{proof}
    We will make use of the following facts about wavelets.
    \begin{itemize}
        \item[(i)] Wavelets $\{\psi_{k,j}\}$ are an unconditional basis
            in $L_p(\mathbb R)$, see.~\cite{Wojt}, \S8.2;
            hence, using~\cite{Wojt}, \S7.3, Corrolary 7.11, we have
            $$
                \|f\|_p \asymp \|(\sum_{k,j}\langle f,\psi_{k,j}\rangle^2 |\psi_{k,j}|^2)^{1/2}\|_p.
            $$
        \item[(ii)] For fixed $k$, inequality~\cite{Wojt}, \S8.1,
            Proposition 8.3 gives
            $$
            \|\sum_{j\in\mathbb Z} \langle f,\psi_{k,j}\rangle\psi_{k,j}\|_p \asymp
            2^{k(1/2-1/p)}\|(\langle f,\psi_{k,j}\rangle)_j\|_{\ell_p}.
            $$
    \end{itemize}

        Use (i), the inequality $(a+b)^r \ge a^r+b^r$ for $r\ge 1$,
        $a,b\ge 0$, and the relation $\|\psi_{k,j}\|_p\asymp
        2^{k(1/2-1/p)}$:
        \begin{multline*}
            \|f\|_p^p \asymp \int |\sum_{k\ge 0,j\in\mathbb Z} \langle
            f,\psi_{k,j}\rangle^2 \psi_{k,j}(t)^2|^{p/2}\,dt \ge
            \sum_{k\ge 0,j\in\mathbb Z}\|\langle
            f,\psi_{k,j}\rangle\psi_{k,j}\|_p^p \asymp \\
            \asymp \sum_{k\ge 0} 2^{kp(1/2-1/p)}\|(\langle f,\psi_{k,j}\rangle)_j\|_{\ell_p}^p.
        \end{multline*}
        It gives us the required lower bound.

    For the upper bound we use (i), put $g_k=(\sum_{j\in\mathbb Z} \langle
        f,\psi_{k,j}\rangle^2 \psi_{k,j}^2)^{1/2}$ and obtain
        $$
            \|f\|_p^2 \asymp \|(\sum_{k\ge 0} g_k^2)^{1/2}\|_p^2 = \|\sum_{k\ge
            0} g_k^2\|_{p/2} \le \sum_{k\ge0} \|g_k^2\|_{p/2} = \sum_{k\ge 0}\|g_k\|_p^2.
        $$
Now, using (i) and then (ii),
\begin{multline*}
    \|g_k\|_p = \|(\sum_{j\in\mathbb Z} \langle f,\psi_{k,j}\rangle^2\psi_{k,j}^2)^{1/2}\|_p \asymp \\
    \asymp \|\sum_{j\in\mathbb Z} \langle f,\psi_{k,j}\rangle\psi_{k,j}\|_p  \asymp
    2^{k(1/2-1/p)}\|(\langle f,\psi_{k,j}\rangle)_j\|_{\ell_p}.
\end{multline*}
Therefore,
$$
\|f\|_p^2 \ll \sum_{k\ge 0} 2^{2k(1/2-1/p)}\|(\langle
        f,\psi_{k,j}\rangle)_j\|_{\ell_p}^2,
$$
as required. Lemma is proved.
    \end{proof}

\begin{proof}[Proof of Theorem~\ref{th_besov}] We will use Theorem A and Lemma 3
    to discretize our problem and then we will bound widths using Theorem 2.
    
    \textbf{Discretization.}
    It is convenient to change coordinates (to get octahedra):
$$
x_{k,j} := 2^{k/2}\langle f,\psi_{k,j}\rangle.
$$
Moreover, as we work with function classes on the segment, we may restrict indices to
the set $T:=\{(k,j)\colon k\ge 0,\,j=1,\ldots,2^k\}$.
    Consider the subspace $L_q\langle\Psi_T\rangle$, spanned by
    $\Psi_T:=\{\psi_{k,j}\colon (k,j)\in T\}$.
    The support of functions from $\Psi_T$ is contained in some segment
    $[-a,a]$; this is also true for $L_q\langle\Psi_T\rangle$.
    The natural projector
    $L_q(\mathbb R)\to L_q\langle\Psi_T\rangle$ is bounded (because wavelets
    form a basis), hence
    \begin{equation}
        \label{discrlow_segment}
    d_n(B^1_{1,\theta}[-a,a],L_q[-a,a])\gg d_n(B^1_{1,\theta}(\R)\cap
    L_q\langle\Psi_T\rangle,L_q\langle\Psi_T\rangle).
    \end{equation}

    \begin{remark}
        One could use periodic wavelets in $\mathcal B^r_{p,\theta}(\mathbb
    R/\mathbb Z)$ instead, see~\cite{Triebel}, \S1.3, Theorem 1.37; however,
        this also brings technical complications.
    The most convenient for us would be Haar wavelets, but they do not form an
        unconditional basis for $\mathcal{B}^1_{1,\theta}$, see~\cite{Ullrich}.
    \end{remark}

We denote by $\ell^\sigma_{p,\theta}(T)$ the analog of the space
$\ell^\sigma_{p,\theta}$ for vectors $x\in\R^T$:
$$
\|x\|_{\ell^\sigma_{p,\theta}(T)} := 
    \begin{cases}
        \left\{\sum_{k\ge 0} 2^{k\sigma\theta}\|(x_{k,j})_{1\le j\le 2^k}\|_{\ell_p}^\theta\right\}^{1/\theta},&\quad \theta<\infty,\\
        \sup_{k\ge 0} 2^{k\sigma}\|(x_{k,j})_{1\le j\le
        2^k}\|_{\ell_p},&\quad \theta=\infty,
    \end{cases}
$$
$b^\sigma_{p,\theta}(T)$ is the unit ball of this space.

Theorem A gives that for $f\in L_q\langle\Psi_T\rangle$ we have
\begin{equation}
    \label{discr_set}
    \|f\|_{\mathcal{B}^1_{1,\theta}} \asymp
    \|\langle f,\psi_{k,j}\rangle\|_{\ell^{1/2}_{1,\theta}(T)} = \|x\|_{\ell^0_{1,\theta}(T)},\quad
    x_{k,j}=2^{k/2}\langle f,\psi_{k,j}\rangle.
\end{equation}
    Indeed, the first term in~(\ref{besov_discr}) is zero as $\psi_{k,j}\perp \varphi_{0,i}$, $k\ge0$
    (this follows from the definition of wavelets and scaling functions).
    So, we have proved that the set $b^0_{1,\theta}(T)$ is the discretization of
    $B^1_{1,\theta}\cap L_q\langle\Psi_T\rangle$.
    The set $b^0_{1,\theta}(T)$ is the infinite-dimensional octahedron for 
$\theta=1$ and the cartesian product of octahedra of dimensions
$1,2,\ldots,2^k,\ldots$ for $\theta=\infty$.

Let us write the lower bound from~(\ref{lq_discr}) for $p=q$, $f\in
    L_q\langle\Psi_T\rangle$, in the coordinates $x_{k,j}$:
\begin{equation}
    \label{discrlow_norm}
    \|f\|_q^q \gg \sum_{k=0}^\infty 2^{-k}\sum_{j=1}^{2^k} |x_{k,j}|^q =
    \|x\|^q_{\ell^{-1/q}_{q,q}(T)}.
\end{equation}

It is important to note that as $q>2$, we have $\|f\|_q\ge \|f\|_2$ and hence $\|f\|_q\ge
\|x\|_{\ell^{-1/2}_{2,2}(T)}$.
Using~(\ref{discrlow_segment}), (\ref{discrlow_norm}) and~(\ref{discr_set}),
we obtain the ``lower'' discretization:
\begin{equation}
    \label{discrlow}
    d_n(B^1_{1,\theta}[-a,a],L_q[-a,a]) \gg
    d_n(b^0_{1,\theta}(T),\ell^{-1/q}_{q,q}(T)\cap\ell^{-1/2}_{2,2}(T)).
\end{equation}
(The change of segment from $[0,1]$ to $[-a,a]$ makes no difference.)

The ``upper'' discretization is analogous.
Pick $f\in  B^1_{1,\theta}[0,1]$. We can extend $f$ from the segment to the
whole line and control it's norm and support.
To do this, one can extend $f$ to
$\R$ (see the Remark after Besov space definition) and then multiply it by some
``cut function'', see~\cite{Besov}, Lemma 3.
We obtain $f\colon\R\to\R$ with support $[-b,b]$. Only finite number of
functions  $\varphi_{0,j}$ have support intersecting $[-b,b]$.
We subtract from $f$ the linear combination of such functions (it is an element
of $n_0$-dimensional space for some fixed $n_0$) and get a function with
support in larger segment $[-B,B]$, but orthogonal to $\psi_{k,j}$, $k<0$
(orthogonality follows from the wavelet definition):
$$
\widetilde{f} := f-\sum \langle\varphi_{0,j},f\rangle \varphi_{0,j}\in
L_q\langle\Psi_+\rangle,\;
\supp\widetilde{f}\subset[-B,B],\;
\|\widetilde f\|_{\mathcal B^1_{1,\theta}(\R)}\ll \|f\|_{\mathcal B^1_{1,\theta}[0,1]}.
$$
Denote
$$
T':=\{(k,j)\colon k\ge0,\,\supp\psi_{k,j}\cap[-B,B]\ne\varnothing\},\quad
\Psi_{T'}:=\{\psi_{k,j}\colon (k,j)\in T'\},
$$
then $\widetilde{f}\in L_q\langle\Psi_{T'}\rangle$.
Therefore
$$
    d_{n+n_0}(B^1_{1,\theta}[0,1],L_q[0,1])\ll d_n(B^1_{1,\theta}(\R)\cap
    L_q\langle\Psi_{T'}\rangle,L_q).
$$
The inequality~(\ref{lq_discr}) in the coordinates $x_{k,j}$ becomes
$$
    \|f\|_q^2 \ll \sum_{k\ge
    0}\left(2^{-k}\sum_{j=1}^{2^k}|x_{k,j}|^q\right)^{2/q} =
    \|x\|^2_{\ell^{-1/q}_{q,2}(T)},
$$
and, finally,
\begin{equation}
    \label{discrup}
    d_{n+n_0}(B^1_{1,\theta}[0,1],L_q[0,1])\ll 
    d_n(b^0_{1,\theta}(T'), \ell^{-1/q}_{q,2}(T'))
\asymp d_n(b^0_{1,\theta}(T), \ell^{-1/q}_{q,2}(T)).
\end{equation}

\textbf{Lower bound.}
We will bound~(\ref{discrlow}), using theorem 2 in the form~(\ref{th_w_form}).
Obviously, $\ell^{-1/q}_{q,q}(T)$ and
$\ell^{-1/2}_{2,2}(T)$ are weighted norms with weights $w_{k,i}=2^{-k}$.

Let us start with the case $\theta=1$; the set $b^0_{1,1}(T)$ is the octahedron.
We denote by $k_0$ and $k_1$ the minimal and maximal $k$, correspondingly, that
satisfy the condition $4n < 2^k < n^{\tilde{q}/2}$, with some fixed
$\tilde{q}\in(2,q)$, e.g. $\tilde{q}=1+q/2$. We apply Theorem 2
and~(\ref{th_w_form}) in the following situation:
\begin{itemize}
    \item from $x\in\R^T$ we take $N=\sum_{k=k_0}^{k_1} 2^k$ coordinates
        $x_{k,i}$, $k_0\le k\le k_1$, $i=1,\ldots,2^k$; the set $b^0_{1,1}(T)$
        becomes $B_1^N$;
    \item weights $w_{k,i}=2^{-k}$, $k_0\le k\le k_1$, $i=1,\ldots,2^k$;
    \item $m=1$ (one block);
    \item $q_1=2$, $q_2=q$.
\end{itemize}

Let us check that the conditions of Theorem~\ref{th_w} are met. Condition 1):
$$
\max w_{k,i}=2^{-k_0} < (4n)^{-1},\quad
\sum w_{k,i} = (k_1-k_0+1)\asymp
\log n.
$$
Condition 2) is satisfied for large enough $n$. Hence from~(\ref{th_w_form}) we have the bound
$$
d_n \gg \min(n^{-1/2}(\log n)^{1/2}, 2^{-k_1/q})\asymp
n^{-1/2}\log^{1/2}n,
$$
as $2^{-k_1/q}>n^{-\frac12\tilde{q}/q}$.

The case $\theta\in(1,2]$ follows from $\theta=1$ as our class grows when
$\theta$ increases.

In the case $\theta=\infty$ we apply Theorem 2 for the same coordinates and
weights, but now the set
$b^0_{1,\infty}(T)$ restricted to $N$ coordinates is the product of octahedra
$\prod_{k=k_0}^{k_1} B_1^{2^k}$ and there are $m\asymp\log n$ blocks $\Delta_s=\{w_{k,i}\colon k=s\}$. 
Arguing as before, we obtain
$$
d_n\gg \min(\sqrt{\frac{\log n}{n}}(\log n)^{1/2}, 2^{-k_1/q})\asymp n^{-1/2}\log n.
$$
We have proven slightly more: if we cut $T$ by $m\asymp\log n$ levels, i.e.,
replace it with
$$
T_m:=\{(k,j)\colon 0\le k< m,\;j=1,\ldots,2^k\},
$$
then the bound is valid:
\begin{equation}
    \label{discrlow_infty}
d_n(b^0_{1,\infty}(T_m),\ell^{-1/q}_{q,q}(T_m)\cap\ell^{-1/2}_{2,2}(T_m))\gg
n^{-1/2}\log n,\quad m\asymp\log n.
\end{equation}

Finally, $\theta\in(2,\infty)$. Use~(\ref{discrlow}) with $T$
replaced by $T_m$, $m\asymp\log n$. Embed $b^0_{1,\infty}(T_m)\subset
m^{1/\theta}b^0_{1,\theta}(T_m)$ and use~(\ref{discrlow_infty}) to get the required
bound.

\textbf{Upper bound.} First we prove that in~(\ref{discrup}) it is enough to
consider $m\asymp \log n$ blocks. Denote $x[k]:=(x_{k,j})_{j=1}^{2^k}$. If $x\in b^0_{1,\theta}(T)$,
then $\|x[k]\|_1\le 1$ for all $k$; hence, $\|x[k]\|_q\le 1$,
$$
\sum_{k\ge m} (2^{-k}\|x[k]\|_q^q)^{2/q} \ll 2^{-2m/q}.
$$
For $m\ge C\log n$ this ``tail'' is less than $n^{-C_1}$; it has no
effect on the width for large $C_1$. So, it is enought to bound
$d_n(b^0_{1,\theta}(T_m),\ell^{-1/q}_{q,2}(T_m))$, $m\asymp \log n$.

We start with the case $\theta=2$. The set of the coordinates is split to blocks
$\Delta_k=\{(k,j)\}_{j=1}^{2^k}$. To approximate in blocks of size $2^k<n/4$ we
take the whole space $\R^{\Delta_k}$, it requires about a half of the dimension
($\le n/2$).
In other blocks we take the subspaces of dimension 
$l\approx n/(2m)$ which are extremal for the widths
$d_{l}(B_1^{2^k},\ell_q^{2^k})\asymp l^{-1/2}2^{k/q}$, as
in~(\ref{octahedra}). We obtain the approximation
$$
\|x[k]-y[k]\|_q\ll (n/m)^{-1/2}2^{k/q}\|x[k]\|_1,
$$
so
$$
(2^{-k}\sum_{j=1}^{2^k}|x_{k,j}-y_{k,j}|^q)^{2/q} \ll (n/m)^{-1}\|x[k]\|_1^2.
$$
Let us sum over $k$:
$$
\|x-y\|_{\ell^{-1/q}_{q,2}(T_m)}^2 \ll (n/m)^{-1} \sum_{k=0}^{m-1} \|x[k]\|_1^2 \le (n/m)^{-1},
$$
as $x\in b^0_{1,2}(T_m)$. So, the error of the approximation is bounded by
$Cn^{-1/2}\log^{1/2}n$.

The case $\theta\in[1,2)$ follows from $\theta=2$ as the class grows as
$\theta$ increases.

If $\theta\ge 2$, we use that $b^0_{1,\theta}(T_m)\subset
m^{1/2-1/\theta}b^0_{1,2}(T_m)$:
$$
d_n(b^0_{1,\theta}(T_m),\ell^{-1/q}_{q,2}(T_m)) \le
m^{1/2-1/\theta}d_n(b^0_{1,2}(T_m),\ell^{-1/q}_{q,2}(T_m))
\ll n^{-1/2}\log^{1-1/\theta}n.
$$

\end{proof}

\section{Future work}

It would be interesting to research other classes related to
$W^1_1$, e.g., Lizorkin--Triebel classes $F^1_{1,\theta}$.

It is possible that one can apply Theorem 2 for widths of classes of functions
of several variables. 

Author thanks reviewer for useful commens on Theorem 2 and A.I.~Tyulenev for
pointing out paper~\cite{Besov}.

\end{document}